\documentclass[11pt,reqno]{amsart}
\usepackage{amsfonts}
\usepackage{graphicx}
\usepackage{amscd}
\usepackage{amsmath,amsfonts,amssymb,amsthm,mathrsfs,xypic}
\xyoption{curve}
\usepackage{fancybox}
\usepackage{hyperref}
\usepackage{mathrsfs}
\usepackage{amsmath}
\usepackage{amsthm}
\usepackage{pstricks-add}
\usepackage[all]{xy}
\usepackage{verbatim}
\usepackage{enumerate}
\usepackage{tikz}
\usetikzlibrary{shapes.geometric, arrows}
\tikzstyle{arrow} = [thick,->,>=stealth]

\newtheorem{MYproposition}{Proposition}
\newtheorem{COR}{Corollary}

\newtheorem{lemma}{Lemma}
\newtheorem{thm}{Theorem}

\theoremstyle{definition}
\newtheorem{MYdef}{Definition}
\newcommand{\Nil}{\operatorname{Nil}}

\newcommand{\DM}{\mathcal D}
\newcommand{\CM}{\mathcal C}
\newcommand{\VM}{\mathcal V}
\newcommand{\End}{\operatorname{End}}

\newcommand{\CMok}{\operatorname{Coker}}
\newcommand{\Ker}{\operatorname{Ker}}
\newcommand{\Hom}{\operatorname{Hom}}
    \newcommand{\HOM}{\operatorname{HOM}}
\newcommand{\Tensor}{\operatorname{Tensor}}

\newcommand{\Mor}{\operatorname{Mor}}

\newcommand{\lra}{\longrightarrow}

\begin{document}
\title{Nilpotent Category of Abelian Category and Self-Adjoint Functors}
\author[Z.W. Bai]{Z.W. Bai}
\author[X. Cao]{X. Cao}

\author[S.T. Mao]{S.T. Mao}

\author[H. Zhang]{H. Zhang}

\author[Y.H. Zhang$^{**}$]{Y.H. Zhang$^{**}$}
\date{}

\address{School of Mathematical Sciences, Shanghai Jiao Tong University,
800 Dongchuan Road, 200240 Shanghai, China}

\email{\hfill{\space}\linebreak Zhiwei\ Bai,\ bai299@sjtu.edu.cn \\
\linebreak Xiang\ Cao,
\ spawner@sjtu.edu.cn
\\ \linebreak Songtao\ Mao,\  jarlly678@sjtu.edu.cn\\ \linebreak Han\ Zhang, \ hanzhang@sjtu.edu.cn \\  \linebreak Yuehui\ Zhang,\ zyh@sjtu.edu.cn}
\thanks{*\ Supported by the NSF of China ( 11671258 and 11771280), and NSF of Shanghai( 17ZR1415400).\\
\indent **\ Corresponding author.}

\begin{abstract}
Let $\mathcal{C}$ be an additive category.   The nilpotent category  $\Nil (\mathcal{C})$ of $\mathcal{C}$, consists of  objects  pairs $(X, x)$ with $X\in\mathcal{C}, x\in\End_{\mathcal{C}}(X)$ such that $x^n=0$ for some positive integer $n$, and a morphism $f:(X, x)\rightarrow (Y,y)$ is $f\in \Hom_{\mathcal{C}}(X, Y)$ satisfying $fx=yf$.  A general theory of $\Nil(\CM)$ is established and it is abelian in the case that $\mathcal{C}$ is abelian. Two abelian categories are  equivalent if and only if their nilpotent categories are equivalent, which generalizes a result in \cite{SWZ}. As an application, it is proved all self-adjoint functors are naturally isomorphic to $\Hom$ and $\Tensor$ functors over the category $\VM$ of finite-dimensional vector spaces. Both $\Hom$ and $\Tensor$ can be naturally generalized to $\HOM$ and $\Tensor$ functor over $\Nil(\VM)$. They are  still self-adjoint, but  intrinsically different.

\end{abstract}
\maketitle
\pagestyle{plain}
\section {\bf Introduction}

Nilpotent operators are one of the most important objects in mathematics. For instance,  the Jordan canonical form of a linear operator over a finite dimensional space can be reduced to that of a nilpotent operator plus a semisimple one, or more generally, the famous Jordan-Chevalley theorem decomposes any linear operator into a semisimple one plus a nilpotent one and these two are commuting with each other. In the past decade, C.M.Ringel and M.Schmidmeier \cite{RS}, D. Simson \cite{S1}, and B.L.Xiong, P. Zhang and Y.H.Zhang \cite{XZZ} and many other mathematicians have developed a beautiful theory using nilpotent operators to the study of the representation theory of artinian algebras, especially finite-dimensional algebras such as $k[x]/(x^n)$, algebras of bounden quivers, various triangular matrix algebras, and so on.

To better understand nilpotent operators as a whole, the category consisting of nilpotent operators deserve a general and thorough study. To this end, consider any additive category $\CM$. Denote by $\Nil(\CM)$ respectively $\Mor(\CM)$ the nilpotent category respectively morphism category of $\CM$ (see section \ref{nilpotent} for the detailed definition). As easily seen that (or cf. \cite{SWZ}) both $\CM$ and the new category $\Nil(\CM)$ are full subcategories of $\Mor(\CM)$.

As Mac Lane's well-known slogan , ``adjoint functors arise everywhere"\cite{M} such as $\Hom$ and $\Tensor$ adjunction. Among all adjunctions,  self-adjunction deserves a deep insight.

 In this paper, a general theory of nilpotent category $\Nil(\CM)$ is established, and it is also abelian in the case that $\mathcal{C}$ is  abelian. Two abelian categories are equivalent if and only if their nilpotent categories are equivalent (Theorem \ref{equivalence}) , which generalizes a result Theorem 3.1 in \cite{SWZ}. As applications, it is proved all self-adjoint functors are naturally isomorphic to $\Hom$ and $\Tensor$ functors over the category $\VM$ of finite-dimensional vector spaces(Theorem \ref{allselfadjoint}). Both $\Hom$ and $\Tensor$ can be naturally generalized to $\HOM$ and $\Tensor$ functor over $\Nil(\VM)$( Definition 1 and Definition 2). They are  still self-adjoint(Theorem \ref{Tensor is self} and Theorem \ref{HOM is self}), but intrinsically different. In particular, they are not adjoint to each other.

Throughout, $1_X$ denote the identity map of an arbitrary object $X$ and we will frequently abbreviate $1=1_X$ if no confusion will occur.

\section{\bf Nilpotent category of an abelian category}\label{nilpotent}
\subsection{ Nilpotent category}
\quad\\
Let $\mathcal{C}$ be any category, the endomorphism category of $\CM$, denoted by $\End({\mathcal{C}})$,  is defined as follows (cf. \cite{SWZ}):

{\rm (1)} an object of $\End(\mathcal{C})$ is a pairs $(X, x)\in\mathcal{C}\times \End_\mathcal{C}(X)$;

{\rm (2)} a morphism $f: (X, x)\longrightarrow(Y, y)$ is a morphism $f\in \Hom_\mathcal{C}(X, Y)$ satisfying $fx=yf$. Intuitively, the morphism $f$ of $\End (\mathcal{C})$ satisfies the following commutative diagram:
\[\xymatrix {
X\ar[d]_{f}\ar[r]^{x} &X\ar[d]^{f}\\
Y\ar[r]_{y} &Y}
\]

If $\CM$ is an abelian category with finite coproduct $\coprod X_i:=\bigoplus X_i$, then $\End(\CM)$ is also abelian in the natural way (cf. \cite{SWZ} Theorem 2.5 for the details) and nilpotent morphisms arise naturally. For $(X, x)\in\End(\CM)$, if there is a positive integer $n=n(x)$ such that $x^n=0$, then $x$ is a nilpotent morphism. Denote by $\Nil(\CM)$ the full subcategory of $\End(\CM)$ consisting of all nilpotent morphisms. We call this category {\em the nilpotent category of} $\CM$.

Obviously $\Nil(\CM)$ is a full subcategory of $\End(\CM)$, so the isomorphism classes of an object $(X,x)$ consisting of those objects $(Y, y)$ with an isomorphism $f: X\lra Y$ such that $y=fxf^{-1}$; that is, the isomorphism classes are indeed conjugacy ones.

Both the category $\End(\CM)$ and $\Nil(\CM)$ are very large for any nontrivial category $\CM$ (see Theorem 2.11 of \cite{SWZ} for the details), so a routine localizing process usually is necessary for reducing the size and complexity. To apply this localization, the full subcategory $\Nil(\CM)$ is required to be a so-called {\em Serre subcategory} or {\em thick subcategory} \cite{S}, that is, a full subcategory closed under extensions. Fortunately, this is really the case for the nilpotent category of $\CM$ as stated in the following.

\begin{thm}\label{Nilabelian}
	Let $\CM$ be an abelian category. Then

$(1)$ $\Hom_{\End(\CM)}((X, 1), (Y,y))=0, \Hom_{\End(\CM)}((Y, y), (X,1))=0$ for all $X, Y\in\CM$ with $y$ nilpotent.

$(2)$ The embedding $X\mapsto (X,0)$ from $\CM$ to $\Nil(\CM)$ is fully faithful.

$(3)$ $\Nil(\CM)$ is a thick abelian subcategory of $\End(\CM)$
\end{thm}
\begin{proof}
The first two statements are direct consequences of the nilpotentness of $y$, so the proofs are left to the reader. Below is a proof of (3). 

Let $f: (X,x)\lra(Y,y)\in\Nil(\CM)$. We proceed to construct the kernel $\Ker(f)$ as follows. Since $\CM$ is abelian, the kernel $g: N\lra X$ of $f: X\lra Y$ in $\CM$ is well-defined, so we have the following commutative diagram:

\[\xymatrix {
N\ar[r]^{g}&X\ar[d]_{x}\ar[r]^{f} &Y\ar[d]^{y}\\
N\ar[r]^{g}&X\ar[r]^{f} &Y
}\]

Now $fx=yf$ forces $fxg=yfg=0$, thus $xg$ factors through $g: N\lra X$, so there exists uniquely an $n: N\lra N$ making the following diagram commutative:

\[\xymatrix {
N\ar[d]^{n}\ar[r]^{g}&X\ar[d]_{x}\ar[r]^{f} &Y\ar[d]^{y}\\
N\ar[r]^{g}&X\ar[r]^{f} &Y
}\]
The above procedure is exactly the same as \cite{SWZ}. Since $gn=xg$, we have $gn^2=xgn=x^2g$, so $gn^{m(x)}=x^{m(x)}g=0$. Thus $n^{m(x)}=0$ since $g: N\lra X$ is the kernel. This
proves that $(N, n)\in\Nil(\CM)$.

To see $g: (N, n)\lra (X,x)$ is the kernel of $f: (X,x)\lra(Y,y)$ in $\Nil(\CM)$, let $h: (Z,z)\lra(X,x)$ satisfy $fh=0$, we need show that $h$ factors through $g$ in $\Nil(\CM)$.

Since $g$ is the kernel of $f$ in the abelian category $\CM$, $h$ factors through $g$ in $\CM$, so $h=gu$ for a unique $u\in{\Hom}_{\CM}(Z, N)$ and we have the following diagram:
$$\xymatrix@R=3mm{
   &  & Z\ar[dd]^<<<<{z}\ar[dr]^{h}\ar@{-->}[dl]_{u} &  & \\
 & N\ar[rr]^<<<<<{\ \ \ g}\ar[dd]_{n} &  & X\ar[r]^{f}\ar[dd]^{x} & Y\ar[dd]^{y} \\
   &  & Z\ar[dr]^{h} \ar@{-->}[dl]_{u} &  &  \\
 & N\ar[rr]^<<<<<{\ \ \ g}&  & X\ar[r]^{f} & Y}$$

$$\xymatrix{
&Y\ar[d]^{p}\\
Z\ar[r]_{f}\ar@{-->}[ur]^{g}&X
}$$
All parallelogram and squares but the left parallelogram are commutative by construction. Since $hz=xh, h=gu, xg=gn$, we have $guz=xgu=gnu$, forcing $uz=nu$ since $g$ is the kernel of $f$. Thus the above diagram commutes and $u$ is nilpotent. Therefore, $g: (N, n)\lra (X,x)$ is indeed the kernel of $f: (X, x)\lra(Y,y)$ in $\Nil(\CM)$.

Dually, the cokernel $(C, c)$ of $f$ exists where $C=\CMok(f)$ in $\CM$ and $c$ is the unique morphism making the following diagram commute:

\[\xymatrix {
X\ar[d]_{x}\ar[r]^{f} &Y\ar[d]^{y}\ar[r]^{\pi}&C\ar[d]^{c}\\
X\ar[r]^{f} &Y\ar[r]^{\pi}&C
}\]

Now we proceed to show the existence of finite coproduct (the product can be constructed dually) in $\Nil(\CM)$. By induction, it is enough to show that any two objects $(X_i, x_i), i=1,2$ have their coproduct. Since $X_1\oplus X_2$ with the natural embeddings $e_i: X_i\lra X_1\oplus X_2$ is the coproduct of $X_1$ and $X_2$ in $\CM$, we claim $(X_1\oplus X_2, x_1\oplus x_2)$ with the natural embeddings $e_i: (X_i, x_i)\lra (X_1\oplus X_2, x_1\oplus x_2)$ is the coproduct of $(X_1, x_1)$ and $(X_2,x_2)$ in $\Nil(\CM)$. To see this, first notice that $x_1\oplus x_2$ is trivially nilpotent. Now let $f_i: (X_i, x_i)\lra (W,w), i=1, 2$ be morphisms in $\Nil(\CM)$, then there exists uniquely a morphism $u: X_1\oplus X_2\lra W$ such that $f_i=ue_i, i=1, 2$, due to the universality of the coproduct. Hence we have the following diagram

$$\xymatrix@R=3mm{
   &  & W\ar[dd]^<<<<{w} &  & \\
 & X_1\ar[rr]^<<<<<{\ \ \ \ \ \beta}\ar[dd]_{x_1}\ar[ur]^{f_1} &  & X_1\oplus X_2\ar[dd]_{x_1\oplus x_2}\ar@{-->}[ul]^{u} & X_2\ar[dd]^{x_2}\ar[l]^{\ \ \ \ \ e_2}\ar[ull]_{\ \ \ f_2} \\
   &  & W &  &  \\
 & X_1\ar[rr]^<<<<<{\ \ \ \ \ \beta}\ar[ur]^{f_1}&  & X_1\oplus X_2\ar@{-->}[ul]^{u} &  X_2\ar[l]^{\ \ \ \ \ e_2}\ar[ull]_{\ \ \ f_2}}$$

We have to show that the above diagram is commutative. Notice that all parallelograms except the one containing $u$ are commutative by construction. To see the parallelogram containing $u$ is commutative, we need show that $wu=u(x_1\oplus x_2)$, which it equivalent to $wue_i=u(x_1\oplus x_2)e_i$ for $i=1, 2$, due again to the universality of the coproduct. Since $f_i=ue_i, f_ix_i=wf_i, e_ix_i=(x_1\oplus x_2)e_i$, we have
$$wue_i=wf_i=f_ix_i=ue_ix_i=u(x_1\oplus x_2)e_i, i=1, 2$$
as required.

The above shows that $\Nil(\CM)$ is a full abelian subcategory of $\End(\CM)$. To see it is also thick in $\End(\CM)$, let $(X, x), (Z, z)\in\Nil(\CM)$ and consider the following commutative diagram with exact rows in $\End(\CM)$:

\begin{equation}\label{thick}
\xymatrix {
0\ar[r]&X\ar[d]_{x}\ar[r]^{f} &Y\ar[d]^{y}\ar[r]^{g}&Z\ar[d]^{z}\ar[r]&0\\
0\ar[r]&X\ar[r]^{f} &Y\ar[r]^{g}&Z\ar[r]&0
}
\end{equation}
We have to show that $(Y, y)\in\Nil(\CM)$. Since $fx=yf$ we know that $y^mf=fx^m$ for all $m\ge0$; similarly, $gy=zg$ implies $gy^m=z^mg$ for all $m\ge0$. Now both $x$ and $z$ are nilpotent, so there is a $t>0$ such that $y^tf=0$, $gy^t=0$. Since the row is exact in the abelian category $\CM$, there are morphisms $p: Z\longrightarrow Y, q: Y\longrightarrow X$ such that $y^t=p g=fq$, due to the universal property of cokernel respectively kernel. Thus $y^{2t}=q gfp=0$, finishing the proof.
\end{proof}
\begin{COR}Let $\CM$ be a $K$-Hom-finite abelian category. Then
$\Nil(\CM)$ is a Serre subcategory of $\End(\CM)$.
\end{COR}
{\em In the remainder of this paper,  $K$ will be a fixed algebraically closed field and $\CM$ be a $K$-Hom-finite abelian category.}

\subsection{The abelian category $\Nil(\CM)$}

In this section, we will assume our abelian category $\CM$ to be Krull–Schmidt and the Grothendieck group $G:=G(\CM)$ of $\CM$ to be of finite rank with a basis consisting of a complete list of nonisomorphic simple objects of $\CM$. It is well known that every nontrivial object in a Krull-Schmidt category decomposes into a finite direct sum of indecomposable objects having local endomorphism rings. By the general theory of abelian categories, all concepts such as projective, injective, indecomposable and simple object are well defined in $\Nil(\CM)$. A remarkable property of $\Nil(\CM)$ inherited from $\End(\CM)$ is that there are neither nonzero projective nor nonzero injective objects no matter whatever $\CM$ does, as stated in the following

\begin{lemma}\label{noprojnoinj}
Let $\CM$ be a K-Hom-finite abelian category. Then there are neither nonzero projective nor nonzero injective objects in  $\Nil(\CM)$.
\end{lemma}

\begin{proof}
It is obvious that (0, 0) is a projective and injective object in $\Nil(\CM)$. We now prove that there are no other injective objects. For any nontrivial object $(X, x)$ in  $\Nil(\CM)$, consider the following morphism:
$$f: (X, x)\longrightarrow (X\oplus X, y)$$
where $f=(1, 0)^T, y=\begin{pmatrix}x&1\\0&x\end{pmatrix}$. Since $x$ is nilpotent, so is $y$. Moreover, $yf=fy=(x, 0)^T$, so $f$ is a $\Nil(\CM)$-morphism and is clearly monic. If $(X, x)$ is an injective object in $\Nil(\CM)$, then $f$ has to be split and there is a morphism $g=\begin{pmatrix}1&z\end{pmatrix}:  (X\oplus X, y)\longrightarrow (X, x)$ such that $gf=1$ and $gy=xg$, forcing $xz-zx=1$, which is impossible over finite dimensional $k$-vector spaces. This shows that $f$ is not split. Therefore, any nontrivial object of $\Nil(\CM)$ is not injective. The dual argument proves that there is no nontrivial projective object in $\Nil(\CM)$.
\end{proof}

Let $0\ne S\in\CM$ and $n\ge2$. Let $S^n$ be the $n$-fold direct sum of $S$ and $J_n$ the nilpotent Jordan block of size $n$. From the proof of Theorem 2.11 of \cite{SWZ}, $(S^n, J_n)$ is also indecomposable nonsemisimple in $\Nil(\CM)$. For convenience, we call such endomorphism of $C^n$ in the form $J_n(\lambda)$ the Jordan-endomorphism of $C^n$, where $\lambda\in{\rm End}_{\CM}(C)$. Indeed, since the isomorphic classes are given by conjugate isomorphisms, that is, $(B, b)$ is isomorphic to $(C, c)$ if and only if $cf=fb$ for some isomorphism $f$ between $B$ and $C$, so all of the  indecomposable objects of $\CM$ are classified by $c=fbf^{-1}$. Thus, isomorphic classes of objects in $\CM$ are indeed given by Jordan blocks $(C^n, J_n(\lambda))$, where $C\in\CM, n$ is a positive integer, and $J_n(c)$ is the Jordan-endomorphism of $C^n$.

Now, we can uncover the relationship between $\Nil(\CM)$ and $\End(\CM)$ in the following

\begin{thm}\label{infinite} Let $\CM$ be a nontrivial K-Hom-finite Krull–Schmidt abelian category with finite rank Grothendieck group $G$. Then 
\begin{itemize}
\item[(1)]$\Nil(\CM)$  is nonsemisimple and of infinite representation type.
    \item[(2)]The full subcategories consisting of simple objects of $\mathcal{C}$ and $\Nil(C)$ are  isomorphic.
    \item[(3)]The Grothendieck groups of $\Nil(C)$ is isomorphic to $G$.
   \item[(4)]The Grothendieck group of $\End(\CM)$ is isomorphic to the direct sum $\coprod_{k\in K}G$.
\end{itemize}
\end{thm}

\begin{proof}(1) follows the proof of Theorem 2.11 of \cite{SWZ}. 

First we prove (2) and (3) will follow. To see the Grothendieck groups of $\CM$ and $\Nil(C)$ are isomorphic, it is enough to prove that $(X, x)\in\Nil(C)$ is simple if and only if $X\in\CM$ is simple. The sufficiency is clear. So suppose $(X, x)\in\Nil(C)$ is simple and $X\in\CM$ is not simple. Then there is a monomorphism $e: Y\longrightarrow X$ in $\CM$ which is not an isomorphism. If $x=0$, then $e$ induces a monomorphism $e: (Y, 0)\longrightarrow (X, 0)$, so $(X, 0)$ is not simple. Suppose the nilpotent morphism $x\not=0$, then  take the nontrivial $Ker(x)$ to be $Y$ and there is a monic morphism $e: Ker(x)\hookrightarrow X$ in $\CM$.  then $e$ again induces a nonisomorphic monomorphism $e: (Y, 0)\longrightarrow (X, 0)$, so $(X, 0)$ is not simple. 

Now we give a proof of (4). Suppose $0\not=B\in\CM$ is not simple, we claim that any object $(B, b)\in\End(\CM)$ is not simple. 

If $b$ is not an automorphism, then there is a monic morphism $f: Ker(b)\hookrightarrow B$ in $\CM$. Thus $f: (\Ker(b), 0)\longrightarrow (B, b)$ is monic in $\End(\CM)$ and $(B, b)$ is not simple in $\End(\CM)$. 

If $b$ is an automorphism, we may assume that $(B, b)$ is indecomposable (otherwise, it is clearly not simple), then its endomorphism ring is local (cf. Proposition 7.4 of \cite{L}, Page 441) and every automorphism is of the form $\lambda 1_B +\eta$, where $\lambda\in K, 1_B$ is the identity of $B$ and $\eta$ is a nilpotent endomorphism of $B$. Now, let $e: S\longrightarrow B$ be any monomorphism starting from a simple object $S$, then $\eta(S)=0$ since $\eta$ is nilpotent. Thus the above $e$ induces a  monomorphism:  $e: (S, \lambda 1_S)\longrightarrow (B, b)$, proving the non-simpleness of $(B, b)$. 

Thus, the only simple objects in $\End(\CM)$ are of the form $(S, 0)$ or $(S, b)$ for simple object $S\in\CM$, where $b$ is an automorphism, which is indeed a scalar multiple, due to the simpleness of $S$. Since any nonzero scalar multiple is only conjugate to itself, we conclude that all the simple objects in $\End(\CM)$ are of the form $(S, k), k\in K$, where $S$ runs over all isomorphic classes of simple objects in $\CM$. Moreover, it is clear that $\Hom_{\End(\CM)}((B, p), (B, q))=0, \Hom_{\End(\CM)}((B, q), (B, p))=0, \forall p\not=q\in K$. Therefore, the Grothendieck group of $\End(\CM)$  is isomorphic to the direct sum of that of $\CM$ with index set $K$. This finishes the proof of (4).

\end{proof}

The natural mapping $\varphi: \Nil(\CM)\longrightarrow \CM$ given by $\varphi(X, x)=X$ and $\varphi(f: (X, x)\longrightarrow (Y, y))=f: X\longrightarrow Y$ defines a faithful dense functor. So we have the following

\begin{thm}\label{equivalence} Let $\CM, \mathcal{D}$ be two abelian categories. The following are equivalent:

$(1)$ $\CM$ is equivalent to $\mathcal{D}$.

$(2)$ $Nil(\CM)$ is equivalent to $\Nil(\mathcal{D})$.

$(3)$ $\End(\CM)$ is equivalent to $\End(\mathcal{D})$.
\end{thm}

\begin{proof}
   The equivalence between (1) and (3) is Theorem 3.1 in \cite{SWZ}. So it suffices to show that (2) is equivalent to (1). It is clear that (1) implies (2). To see the inverse, let $F: \Nil(\CM)\longrightarrow\Nil(\mathcal{D})$ be an equivalence. Then the restriction $F|_{(\CM, 0)}$ of $F$ on the full subcategory $(\CM, 0)=\{(X, 0): X\in\CM\}$ of $\Nil(\CM)$ gives an equivalence between $(\CM, 0)$ and its image $F(\CM, 0)$. Claim the image $F(\CM, 0)=(\mathcal{D}, 0)$: Since $F$ is an equivalence, every $(D, 0)\in \Nil(\mathcal{D})$ is isomorphic to some $F(X,x)=(Y, y)\in \Nil(\mathcal{D})$. Denote the isomorphism by $\phi$. Then $\phi$ gives an isomorphism between $D$ and $Y$, and $y\phi=\phi=0$, forcing $y=0$, thus $(Y, y)\cong (D, 0)$ is in the image of $F$. This further induces an equivalence between $\CM$ and $\mathcal{D}$, since the natural functor $(X, 0)\mapsto X, f\mapsto f$ is clearly an equivalence of $(\CM, 0)$ and $\CM$.
\end{proof}

{\it Remark.} One hopes naturally to transfer the whole setup of $\Nil(\CM)$ in the present paper on a triangulated category. Unfortunately, this does not work properly. $\Nil({\mathcal T})$ does not have a natural triangulated category structure for a  triangulated category  $\mathcal{T}$, the most important obstruction of this failure is the lack of uniqueness of some suitable morphism in the third axiom of a triangulated category, the detail can be found in Remark 2.12 and Example 2.13 of \cite{SWZ}.

\begin{COR}
\label{Moritaequivalence}Let $A, B$ be two finite-dimensional algebras over $k$ with $\CM, \mathcal{D}$ their finite-dimensional (left) module categories respectively. Then following are equivalent:

$(1)$ $A$ is Morita equivalent to $B$.

$(2)$ $\CM$ is equivalent to $\mathcal{D}$.

$(3)$ $Nil(\CM)$ is equivalent to $\Nil(\mathcal{D})$.

$(4)$ $\End(\CM)$ is equivalent to $\End(\mathcal{D})$.
\end{COR}

\begin{proof} The equivalence of (1), (2) and (4) is Theorem 3.1 in \cite{SWZ}. Now, Theorem \ref{equivalence} guarantees the equivalence of (4) and the other three conditions.
\end{proof}

 \section{\bf{Application: Self-Adjoint Functors}}

In this section, we give one application of the theory developed in the previous sections. That is the self-adjoint functors over $\CM$ and $\Nil(\CM)$, where $\CM$ is the category of finite-dimensional vector spaces over the algebraically closed field $K$.

Let $\CM, \DM$ be two categories and $F:\CM\longrightarrow\DM$ ,  $G:\DM\longrightarrow\CM$ be two functors. Then $F$ (respectively $G$) is called to be a left (respectively right) adjoint of $G$ (respectively $F$) if there is a natural bijection:
\begin{equation}\label{adjoint}
    \Hom_\DM(F(X), Y) \longrightarrow \Hom_\CM(X, G(Y)), \ \ \forall X\in\CM, Y\in\DM
\end{equation}

If $\CM=\DM$ and $F=G$, the most interesting case of adjoint arises, namely self-adjoint, which means that a functor left (hence also right) adjoint to itself. 

We recall the definition of self-adjoint functors.

	A functor $F:\CM \longrightarrow \CM$ is called to be {\em self-adjoint} if there exists a natural bijection :
	\begin{equation*}\label{C adjoint}
	    \varphi_{X,Y}:\Hom_{\CM}(F(X),Y) \rightarrow \Hom_{\CM}(X, F(Y))
	\end{equation*}
	for any objects $X, Y\in \CM$. The word `natural' means that $\forall f:X\rightarrow X'$, the following  diagram commutes.
\[\xymatrix{
\Hom(F(X), Y)\ar[r]^{\varphi_{X,Y}}\ar[d]&\Hom(X,F(Y))\ar[d]\\
\Hom(F(X'), Y)\ar[r]^{\varphi_{X',Y}}&\Hom(X', F(Y))
}\]

In the sequel, $\VM$ will be the category of finite-dimensional vector spaces over $K$ in the remainder. 

\subsection{Self-adjoint Functors over $\VM$}

It is  well-known that $\Hom$ and $\Tensor$ functors are adjoint over many abelian categories. In particular, they are isomorphic over $\VM$. Furthermore, they are the unique self-adjoint functor over $\VM$  up to isomorphism , as proved in the following

\begin{thm}\label{allselfadjoint}
Let $F$ be a self-adjoint functor over $\VM$. Then it is naturally isomorphic to $\Hom$ functor: $\Hom(F(K),-)$ and $\Tensor$ functor: $-\otimes F(K) $.
\end{thm}
\begin{proof}We prove for $\forall f\in \Hom (A, A')$,
 there exists a natural isomorphism
$$\eta: \Hom(F(K),-) \longrightarrow F$$
 such that the following diagram commutes. (Denote $\delta=\Hom(F(K), f)$).
\[\xymatrix{
\Hom(F(K), A)\ar[r]^-{\eta_{A}}\ar[d]_{\delta}&F(A)\ar[d]^{F(f)}\\
\Hom(F(K), A')\ar[r]^-{\eta_{A'}}&F(A')
}\]
Since $F$ is self-adjoint, there is a natural bijection $\varphi: \Hom(F(K),-)\longrightarrow \Hom(K, F(-))$. Define $\psi: \Hom(K,F(-)) \longrightarrow F$ by $\psi_A(g)= g(1), \forall g\in \Hom(K, F(A)),$ where $1$ is the unit element of $K$. It is well-known $\psi$  is also a bijection.  
\[\xymatrix{
    \Hom(F(K), A)\ar[r]^-{\varphi_A}\ar[d]_{\delta}&\Hom(K, F(A))\ar[r]^-{\psi_A}\ar[d]_{\Hom(K, F(f))}&F(A)\ar[d]_{F(f)}\\
\Hom(F(K), A')\ar[r]^-{\varphi_{A'}}&\Hom(K, F(A'))\ar[r]^-{\psi_{A'}}&F(A')
}\]
Since $[F(f)g](1)=F(f)g(1)=F(f)\psi_A(g)$, then $\psi$ is natural, namely the right square above commutes. 
Hence the two small squares commute, forcing the big square commutative. Therefore, $\eta:= \psi\varphi$ is a desired natural isomorphism.
\end{proof}
Surprisingly, $\Tensor$ and $\Hom$ functors are completely different over $\Nil(\VM)$, which will be clarified in the next subsection.

\subsection{Self-adjoint Functors over $\Nil(\VM)$}\quad\\
To understand the self-adjoint functor over $\Nil(\VM)$, we fix some notations: $K^p$ is the usual $p$-dimensional vector space, and $J_p$ is the nilpotent transformation of $K^p$ whose matrix under the standard basis is the $p$-dimensional Jordan block. Define $\dim(X,x):=\dim(X)$.

\begin{MYproposition}\label{indecomposable1}
\label{Hom-dimension}
\label{mension}
\begin{itemize}
    \item[(1)]$(X,x)\in \Nil (\VM)$ is indecomposable if and only if $(X,x)\cong (K^p,J_p)$.
    \item[(2)]Let $p, q$ be two positive integers. Then
$$\dim \Hom((K^p, J_p), (K^q, J_q))= \min\{p, q\}.$$
Furthermore, $\forall (X,x)=\bigoplus_{i=1}^s (K^{p_i},J_{p_i})$, $(Y,y)=\bigoplus_{j=1}^t (K^{p_j},J_{q_j}) \in \Nil (\VM)$.
Then
$$\dim \Hom((X,x),(Y,y))=\sum_{1\le i\le s,1\le j\le t} \min(p_i,q_j).$$
\end{itemize}

\end{MYproposition}

\begin{proof}By definition, $f\in\Hom((K^p, J_p), (K^q, J_q))$ if and only if $f\in\Hom(K^p,K^q)$ satisfies 
	$$ J_{q} f = f J_p.$$
Under a certain basis, this is in fact a {\it Lyapunov matrix equation} of the form $MX-XN=0$ with unknown matrix $X$. It is well known that such a Lyapunov equation is equivalent to the following system of linear equations in unknown vector $x$
\begin{equation}\label{lyapunov}(M^T\otimes I - I\otimes N)x=0
\end{equation}\label{Coefficient Matrix}
As to our case, $M=J_q, N=J_p$, so  the coefficient matrix  in the equation (\ref{lyapunov}) is 
in the following partitioned form 
\[J_q^T\otimes I_p - I_q\otimes J_p=\begin{pmatrix}
-J_q&0&0&\cdots&0&0\\
I_q&-J_q&0&\cdots&0&0\\
0&I_q&-J_q&\cdots&0&0\\
\vdots&\vdots&\ddots&\ddots&\vdots&\vdots\\
0&\cdots&0&I_q&-J_q&0\\
0&\cdots&0&0&I_q&-J_q
\end{pmatrix}_{pq\times pq}.\]

It is straight-forward to calculate the rank of $J_q^T\otimes I_p - I_q\otimes J_p$
\begin{equation}\label{rank}r(J_q^T\otimes I - I\otimes J_p)=pq-\min\{p,q\}.
\end{equation}
So
\begin{equation*}
   \begin{aligned}
       \dim\Hom((K^p, J_p), (K^q, J_q))=&\; pq-(pq-\min\{p, q\})\\
       =&\min\{p, q\}.\\
    \end{aligned}
\end{equation*}
\end{proof}

The following lemma is easy to check. 
\begin{lemma}
\label{basis}
Given $\Hom((K^p, J_p), (K^q, J_q))$, there is a basis $\{f_i\}$ $(i=1,\cdots,\min\{p,q\})$ of $\Hom((K^p, J_p), (K^q, J_q))$ satisfying
$$J_q f_i=f_{i+1},\quad i=1,\cdots, \min\{p,q\}-1$$
\end{lemma}

Now we generalize $\Hom$ and $\mathrm{Tensor}$ to  $\Nil(\VM)$ as follows. 
\begin{MYdef}
Let $B\in \VM$ be fixed, $b\in \End_\VM(B)$ invertible.
        The $\Tensor$ functor $-\otimes(B, b) $: $\Nil(\VM)\longrightarrow\Nil(\VM)$ consists of 
        \begin{itemize}
            \item a function: $\mathrm{ob}(\Nil(\VM))\longrightarrow \mathrm{ob}(\Nil(\VM)),$
            $$(X,x)\longmapsto(X\otimes B, x\otimes b);$$

            \item for each $(X,x), (Y,y)\in \Nil(\VM)$, a function :
            $\Hom((X,x),(Y,y))\longrightarrow\Hom((X\otimes B, x\otimes b),(Y\otimes B, y\otimes b)),$
             $$f\longmapsto f\otimes 1_B.$$
        \end{itemize}

\end{MYdef}
It is easy to check that $\Tensor$ functor is well-defined.
     
{\it Remark.} 
     When the nilpotent morphism is taken to be 0, i.e.$(X,x)=(X,0)$, then
$$\Hom_{(\Nil\VM)}((Y,0),(X,0)\otimes  (B, b)) = \Hom_{(\Nil\VM)}((Y,0),(X\otimes B, 0\otimes b))=\Hom_{\VM}(Y, X\otimes B).$$  So it is a natural generalization of the Tensor functor over $\VM$.
\begin{thm}\label{Tensor is self}
The ${\Tensor}$ functor: $-\otimes (B, b)$ is self-adjoint.
\end{thm}
\begin{proof}Denote $\dim B=d$.
Since $x$ is nilpotent,  $x\otimes b$ is  conjugate to $x\otimes 1_B$ due to the Jordan canonical form of a Kronecker product theory, so $(X\otimes B, x\otimes b)$ is indeed isomorphic to $\bigoplus_d (X,x)$.

We need to prove that for $\forall (X,x),(Y,y)$, there exists a bijection $\varphi$ such that  $\forall f: (X', x')\rightarrow (X, x)$, the following diagram commutes
\[\xymatrix{
\Hom(\bigoplus_d (X,x), (Y,y))\ar[r]^-{\varphi_{x,y}}\ar[d]&\Hom( (X,x), \bigoplus_d(Y,y))\ar[d]\\
\Hom( \bigoplus_d (X',x'), (Y,y))\ar[r]^-{\varphi_{x',y}}&\Hom( (X',x'),  \bigoplus_d(Y,y))
}\]
$\forall g=(g_i)\in \Hom(\bigoplus_d (X,x), (Y,y)), g_i\in \Hom((X,x), (Y,y)), 1\leq i\leq d$ . Define $\varphi_{x,y}(g)=(g_1, g_2, \cdots, g_d)^T$. Obviously $\varphi$ is a bijection, and $(g_1, g_2, \cdots, g_d)^T f=(g_1 f, g_2 f, \cdots, g_d f)^T$, implying the above commutative diagram.
\end{proof}

\begin{MYdef}\label{HOMdef}
   Fix an object $(A, a)\in\Nil(\VM)$. $\HOM$ functor ${\HOM}((A, a), -)$: $\Nil(\VM)\longrightarrow\Nil(\VM)$ consists of 
        \begin{itemize}
            \item a function: $\mathrm{ob}(\Nil(\VM))\longrightarrow \mathrm{ob}(\Nil(\VM))$
            $$(X,x)\longmapsto (\Hom((A, a), (X,x)),\theta_x),$$ where $\theta_x(f)=fx,\quad\forall f \in \Hom((A,a),(X,x))$.
            
            \item for each $(X,x), (Y,y)\in \Nil(\VM)$, a function:
            $$\Hom((X,x),(Y,y))\longrightarrow \Hom(\Hom((A,a),(X,x)),\Hom((A,a),(Y,y)))$$
             $$f\longmapsto \Hom((A,a),f).$$
             
        \end{itemize}
 
 \end{MYdef}
 
 It is easy to check $\HOM$ is well-defined.
 
 {\it Remark.} Note that
 \begin{equation*}
     {\HOM}((A, 0), -),(X,0))=\Hom_{(\Nil\VM)}((A, 0), (X,0),\theta_0)=\Hom_{\VM}(A,X).
 \end{equation*}
So, $\HOM((A, 0), -)$ over $\Nil(\VM)$ is a natural generalization of $\Hom(A, -)$ over $\VM$.

This is indeed the usual definition of $\Hom$ functor over $\VM$, so it is a natural generalization. But it is not isomorphisc to $\Tensor$ functor(see below), we call it `HOM' functor.

In order to prove the $\HOM$ functor is self-adjoint, we need the following lemmas.
\begin{lemma}\label{important}
For two indecomposable objects $(K^p,J_{p})$ and $(K^q,J_{q})$,
$$(\Hom((K^p,J_p),(K^q,J_q)),\theta_{J_q})\cong ( K^{\min\{p,q\}},J_{\min\{p,q\}})$$
\end{lemma}
\begin{proof} 
We just need to prove that $J_{\min\{p,q\}}$ and $\theta_{J_q}$ are conjugated, that is, there is an isomorphism $f\in\Hom(\Hom((K^p,J_p),(K^q,J_q)), K^{\min\{p,q\}})$ such that $\theta_{J_q}=f^{-1}J_{\min\{p,q\}}f$, equivalently $$f\theta_{J_q}=J_{\min\{p,q\}}f$$

We suppose a basis $\{e_i\}$ $(i=1,2,\cdots, \min\{p, q\})$ of $K^{\min\{p,q\}}$ and by Lemma \ref{basis},  there is a basis $\{f_i\}$ $(i=1,2,\cdots, \min\{p,q\})$ of $\Hom((K^p, J_p), (K^q, J_q))$ satisfying the following condition
$$J_{q}f_i=f_{i+1},\quad i=1,\cdots, \min\{p,q\}-1$$
We directly construct bijective $f$ mapping $f_i$ to $e_i$ $(i=1,2,\cdots, \min\{p, q\})$.

For any $\phi \in  \Hom((K^p,J_p),(K^q,J_q))$, we show that $f\theta_{J_q}(\phi)=J_{\min\{p,q\}}f(\phi)$:

Suppose $\displaystyle\phi=\sum_{i=1}^{\min\{p,q\}}a_if_i$. Then by corollary \ref{basis}
$$
\begin{aligned}
f\theta_{J_q}(\phi)=&f\theta_{J_q}(\sum_{i=1}^{\min\{p,q\}}a_i f_i)=f(\sum_{i=1}^{\min\{p,q\}}a_i\theta_{J_q}f_i)= f(\sum_{i=1}^{\min\{p,q\}}a_i J_q f_i)\\=&f(\sum_{i=1}^{\min\{p,q\}-1}a_i f_{i+1})=\sum_{i=1}^{\min\{p,q\}-1}a_i e_{i+1}=J_{\min\{p,q\}}(\sum_{i=1}^{\min\{p,q\}}a_i e_i)\\=&J_{\min\{p,q\}}f(\sum_{i=1}^{\min\{p,q\}}a_i f_i)=J_{\min\{p,q\}}f(\phi).
\end{aligned}
$$
\end{proof}

\begin{thm} \label{HOM is self}
${\HOM}$ functor ${\HOM}((A,a), -)$ is self-adjoint.
\end{thm}
\begin{proof}
It is enough to prove the theorem for indecomposable objects of $\Nil\VM.$ So let $(A,a), (X,x), (Y,y)$ be indecomposable objects. Denote $p=\dim X$, $q=\dim Y$, $r=\dim A$.

First,  by Lemma \ref{important}, 
$$\dim\Hom((\Hom((A,a),(X,x)),\theta_{x}) ,(Y,y))=\min\{\min\{
r,p\},q\}$$
and 
$$\dim\Hom((Y,y),(\Hom((A,a),(Y,y)),\theta_{y}))=\min\{p,\min\{r,q\}\}$$
Thus 
\begin{equation}\label{equal}
    \Hom((\Hom((A,a),(X,x)),\theta_{x}) ,(Y,y))\cong \Hom((Y,y),(\Hom((A,a),(Y,y)),\theta_{y})).
\end{equation}

Second, we proceed to prove $\HOM((A,a),-)$ is self-adjoint, i.e., there exists a natural bijection $\varphi
:\Hom(-,\HOM((A,a),-))\longrightarrow\Hom(\HOM((A,a),-), -)$.

We define $\varphi$ by two steps. First, fix $\beta\in A$, $\forall g\in\Hom((X,x),\Hom((A,a),(Y,y),\theta_y)),$ 
 define $h :X\rightarrow Y$ by
$h(\alpha) = g(\alpha)(\beta), \forall \alpha\in X$.  Since $ g(x(\alpha))=\theta_y(g(\alpha))=yg(\alpha)$, then $hx(\alpha)=g(x(\alpha
))(\beta)=y(g(\alpha)(\beta))=yh(\alpha)$, thus  $hx=yh$. So $h\in \Hom((X,x),(Y,y))$.

Now for $ \rho \in \Hom((A,a),(X,x)),$  we define $\varphi: \Hom((A,a),(X,x))\rightarrow Y$ by $\varphi(g)(\rho)=h\rho(\beta)$. Denote $\tilde{g}=\varphi(g)$. Since $\tilde{g}\theta_x(\rho)=\tilde{g}(x\rho)=hx\rho(\beta)$, $y\tilde{g}(\rho)=y(h\rho(\beta))=yh\rho(\beta)$ and  $hx=yh$, then $\tilde{g}\theta_x(\rho)=y\tilde{g}(\rho)$, thus $\tilde{g}\theta_x=y\tilde{g}$. So $\tilde{g}\in \Hom(\HOM((A,a),(X,x)),(Y,y))$.

We prove $\varphi$ is injective. Suppose $0\neq g\in \Hom((X,x), \Hom((A,a),(Y,y), \theta_y))$ satisfying  $\varphi(g)=0$.  Then
 \begin{equation}\label{sasa}
     0=\varphi(g)(\rho)=h \rho(\beta)=g(\rho(\beta))(\beta), \forall \rho\in \Hom((A,a),(X,x)).
 \end{equation}
 
 Since $g$ is non-zero, there exists $u\in X$ satisfying $g(u)\neq 0$. So we can choose $\beta
 $ not in the kernel of $g(u)$. To produce a desired contradiction to equation (\ref{sasa}), we construct $\rho\in\Hom((A,a),(X,x))$ by
 $$\rho(a^{i-1}(\beta))=x^{i-1}(u),\quad 1\le i\le s,$$
 where $s$ is the nilpotent index of $a$.
 
 By Lemma \ref{basis}, $\{\beta,a(\beta),\cdots,a^{s-1}(\beta)\} $ is a basis of $A$. $\forall v\in  A$, $v=\sum_{i=1}^s t_i a^{i-1}(\beta).$
Then
$$
\begin{aligned}
x\rho(v)=&x\rho(\sum_{i=1}^s t_ia^{i-1}(\beta))=x(\sum_{i=1}^s t_ix^{i-1}(\beta))=\sum_{i=1}^s t_i x^{i}(\beta)\\=& \rho(\sum_{i=1}^s t_i a^{i}(\beta))=\rho a(\sum_{i=1}^s t_i a^{i-1}(\beta))=\rho a(v).
\end{aligned}$$

Back to equation (\ref{sasa}), 
$g(\rho(\beta))(\beta)=g(u)(\beta)=0$, 
 contradicting to the choice of $\beta$. Thus $\varphi$ is injective. 

By equation(\ref{equal}), $\varphi$ is an isomorphism.

Finally,  $\forall f:(X',x')\rightarrow(X,x)$,
  $ \mu\in \Hom((A,a),(X',x'))$, we have   $\tilde{g}\HOM(f)(\mu)=\tilde{g}(f(\mu))=h(f\mu(\beta))=\widetilde{g f}(\mu)$. So $\tilde{g}\HOM(f)=\widetilde{g f}$, namely the following diagram commutes
 \[\xymatrix{
\Hom((X,x),\HOM((A,a),(Y,y)))\ar[d]\ar[r]^{\varphi_{x,y}}&\Hom(\HOM((A,a),(X,x)), (Y,y))\ar[d]\\
\Hom((X',x'),\HOM((A,a),(Y,y)))\ar[r]^-{\varphi_{x',y}}&\Hom(\HOM((A,a),(X',x')), (Y,y))
}\]

Therefore, $\varphi$ is a natural isomorphism, and $\HOM$ is self-adjoint.

This finishes the proof. 
\end{proof}

{\it Remark.} 
Although $\Hom$ and $\Tensor$ functors are the unique self-adjoint functors over $\VM$ due to Theorem \ref{allselfadjoint}, each $\Tensor$ functor $-\otimes(B,b)$ and $\HOM$ functor $\HOM((A,a),-)$, where $(A,a)=\bigoplus_{i=1}^s(K^{p_i}, J_{p_i})$, are intrinsically different over $\Nil(\VM)$ by Theorem \ref{Tensor is self} and Theorem \ref{HOM is self}, consequently not adjoint to each other since they are both self-adjoint: For arbitrary indecomposable object $(X,x)\in\Nil(\VM)$, $\dim(X,x)\otimes(B,b)=\dim X\cdot\dim B$, but  $\dim\HOM((A,a),(X,x))=\sum_{i=1}^s\min\{\dim X, p_i\}$. The former is the linear function of $\dim X$, but the latter is not.

\end{document}